\documentclass{amsart}

\usepackage{amsmath,amssymb,amsfonts,enumerate,amsthm,graphicx,color,hyperref}

\newcommand{\tn}{\textnormal}
\newcommand{\CC}{\mathcal{C}}
\newcommand{\WW}{\mathcal{W}}
\newcommand{\KK}{\mathcal{K}}
\newcommand{\se}{\subseteq}
\newtheorem{thm}{Theorem}[section]
\newtheorem{cor}[thm]{Corollary}
\newtheorem{lem}[thm]{Lemma}
\newtheorem{prop}[thm]{Proposition}

\newtheorem{exam}[thm]{Example}
\newtheorem{rem}[thm]{Remark}

\numberwithin{equation}{section}

\begin{document}
\bibliographystyle{amsplain}


\author{Amir Mousivand}
\address{Amir Mousivand\\Department of Mathematics, Islamic Azad
University, Firoozkooh branch, Firoozkooh, Iran.}\email{amirmousivand@gmail.com\\amir.mousivand@iaufb.ac.ir}



\title{Algebraic properties of product of graphs}

\begin{abstract} Let $G$ and $H$ be two simple graphs and let $G*H$ denotes the graph theoretical product of $G$ by $H$. In this paper we provide some results on graded Betti numbers, Castelnuovo-Mumford regularity, projective dimension, $h$-vector, and Hilbert series of $G*H$ in terms of that information of $G$ and $H$. To do this, we will provide explicit formulae to compute graded Betti numbers, $h$-vector, and Hilbert series of disjoint union of complexes. Also we will prove that the family of graphs whose regularity equal the maximum number of pairwise $3$-disjoint edges, is closed under product of graphs.
\end{abstract}

\maketitle

\section{Introduction}
To any finite simple graph $G$ with the vertex set $V(G)=\{x_1,...,x_n\}$ and the edge set $E(G)$, one can attach an ideal in the polynomial rings $R=\mathbb{K}[x_1,...,x_n]$ over the field $\mathbb{K}$, whose generators are square-free quadratic monomials $x_ix_j$ such that $\{x_i,x_j\}$ is an edge of $G$. This ideal is called the {\it edge ideal} of $G$ and will be denoted by $I(G)$. Also the {\it edge ring} of $G$, denoted by $\mathbb{K}[G]$ is defined to be the quotient ring $\mathbb{K}[G]=R/I(G)$. Edge ideals and edge rings were first introduced by Villarreal \cite{V1} and then they have been studied by many authors in order to examine their algebraic properties according to the combinatorial data of graphs. The most important Algebraic objects among these are Betti numbers, projective dimension, (Castelnuovo-Mumford) regularity, $h$-vector and Hilbert series. The aim of this paper is to investigate the above mentioned algebraic properties of product of graphs according to the information of the original graphs.

For any two finite simple graphs $G=(V(G),E(G))$ and $H=(V(H),E(H))$ over disjoint vertex sets (i.e. $V(G)\cap V(H)=\emptyset$), the graph theoretical {\it product of $G$ by $H$}, denoted by $G*H$, is the graph over the vertex set $V(G)\cup V(H)$ whose edge set is
$$E(G*H)=E(G)\cup E(H)\cup\{\{x,y\}~|~x\in V(G) ~ \tn{and} ~ y\in V(H)\}.$$

Our first topic in this paper is about the regularity of product. The regularity is one of the most important invariants of a graded module so that the regularity of edge rings has been a subject of study in combinatorial commutative algebra (see for example, \cite{Z,FHVT,VT,Ku,MMCRTY}). The (Castelnuovo-Mumford) regularity of a graded $R$-module M, denoted by $\tn{reg}(M)$, is defined by
$$\tn{reg}(M) = \max\{j-i~|~\beta_{i,j}(M)\neq 0 \}.$$
Whieldon in \cite{W} determined the graded Betti numbers of the edge ring $\mathbb{K}[G*H]$ (see \cite[Lemma 5.4]{W}). Using this we will reprove known results on the Betti numbers of some families of graphs such as complete bipartite graphs, wheels, and star graphs. In addition, we will prove that the regularity of product is maximum of the regularity of the original graphs. Namely:

\begin{prop} Let $G$ and $H$ be two simple graphs with disjoint vertex sets. Then
$$\tn{reg}(R/I(G*H))=\tn{max}\{\tn{reg}(R/I(G)),\tn{reg}(R/I(H))\}.$$
\end{prop}

One of the useful invariant of a graph that relates to the regularity is the number of pairwise 3-disjoint edges of $G$.
Two edges $\{x, y\}$ and $\{u, v\}$ of a graph $G$ is called {\it 3-disjoint} if the induced
subgraph of $G$ on $\{x,y,u,v\}$ is disconnected. A set $\Gamma$ of edges of $G$ is called a {\it pairwise
3-disjoint set of edges} if any pair of  edges of $\Gamma$ is 3-disjoint.The maximum cardinality of all pairwise 3-disjoint
sets of edges in $G$ is denoted by $a(G)$.

Katzman provided the following lower bound of the regularity of the edge ring.
\begin{thm} \cite[Lemma 2.2]{Ka} \label{regT} For any graph $G$, $\tn{reg}(R/I(G))\geq
a(G)$.
\end{thm}

The following natural question arises: are there any families of graphs where this inequality is an equality? There has been several attempts to answer this question.
Zheng \cite{Z} proved the equality for trees.
Francisco, H\`a and Van Tuyl \cite{FHVT} proved equality holds
for Cohen-Macaulay bipartite graphs. Van Tuyl \cite{VT} generalized this to the family of sequentially Cohen-Macaulay bipartite graphs. Note that a tree is a sequentially Cohen-Macaulay bipartite graph (\cite{F}).
Kummini \cite{Ku} proved equality holds also for unmixed bipartite graphs.
In addition, the authors in \cite{MMCRTY} generalized Kummini's  result to the class of very well-covered graphs.

Let $\mathcal{A}$ be the set of all graphs whose regularity equal the maximum number of pairwise 3-disjoint edges, i.e.
$$\mathcal{A}=\{~G~ ~ ~|~ ~ ~ G ~\tn{is a simple graph with } \tn{reg}(R/I(G))=a(G)\}.$$
Therefore $\mathcal{A}$ contains the above mentioned classes of graphs. We prove $\mathcal{A}$ is closed under product of graphs. Namely:

\begin{prop} Let $\mathcal{A}$ be the set of all graphs $G$ with the property $\tn{reg}(R/I(G))=a(G)$. Then $\mathcal{A}$ is closed under product of graphs, i.e. for any $G,H\in \mathcal{A}$ one has
$$\tn{reg}(R/I(G*H))=a(G*H).$$
\end{prop}

Our second topic is the $h$-vector and Hilbert series of product. $h$-vectors of simplicial complexes has actively been studied in both viewpoint  of combinatorics and combinatorial commutative algebra so that $h$-vectors of some families of simplicial complexes have been completely characterized (see \cite{S2}). A classical result of Stanley guarantees that $h$-vectors of Cohen–Macaulay complexes are non-negative (see \cite{S1,BF,BR,MNZ,MT} for more results on $h$-vector).

Let $\Delta$ be a simplicial complex over a finite set of vertices. One of the fundamental invariants of $\Delta$ is its $f$-vector, $f(\Delta)=(f_{-1},f_0, \ldots , f_d)$, where each $f_i$ is the number of faces of $\Delta$ with dimension $i$ and $d=\tn{dim}(\Delta)$. More algebraically, another important invariant of $\Delta$ is the $h$-vector, $h(\Delta)=(h_0,h_1,\ldots,h_{d+1})$, which can be found from the Hilbert series of the Stanley-Reisner ring of $\Delta$, i.e.
$$\sum_{i=0}^{d+1}h_it^i=\sum_{i=0}^{d+1}f_{i-1}t^i(1-t)^{d+1-i}.$$
The two sequences are interchangeable via the equations
$$h_k=\sum_{i=0}^k(-1)^{k-i}\binom{d+1-i}{k-i}f_{i-1},$$
and
$$f_{k-1}=\sum_{i=0}^k\binom{d+1-i}{k-i}h_i.$$

In Section 4 we will express the $h$-vector and Hilbert series of disjoint union of complexes in terms of $h$-vector and Hilbert series of the original complexes, respectively. Indeed, we prove the following two propositions.

\begin{prop} Let $\Delta$ and $\Delta'$ be two simplicial complexes with disjoint vertex sets. Also let $\tn{dim}(\Delta')\leq \tn{dim}(\Delta)$, and assume $n=\tn{dim}(\Delta)-\tn{dim}(\Delta')$. Then
$$h_k(\Delta\cup\Delta')=h_k(\Delta)+\sum_{p=0}^{n}(-1)^p\binom{n}{p}h_{k-p}(\Delta')-(-1)^k\binom{d+1}{k},$$
for all $0\leq k \leq d+1$, where $d=\tn{dim}(\Delta)$.
\end{prop}

\begin{prop} Let $\Delta$ be a simplicial complex and let $\Delta_1,\ldots,\Delta_r$ be connected components of $\Delta$. Then
$$H_{\mathbb{K}[\Delta]}(t)=\sum_{j=1}^rH_{\mathbb{K}[\Delta_j]}(t)-(r-1).$$
\end{prop}

Applying the above result to the independence complex of graphs enables us to provide the following results on the Hilbert series of product of some families of graphs.

\begin{cor} Let $G$ be a simple graph and $\KK_m$ be the complete graph over $m$ vertices. Then
$$H_{\mathbb{K}[G*\KK_m]}(t)=H_{\mathbb{K}[G]}(t)+m\frac{t}{1-t}.$$
\end{cor}

\begin{cor} Let $G$ be a simple graph and let $\overline{\KK}_m$ denotes the complement of $\KK_m$. If $m\leq \tn{dim}(\Delta)+1$,  then
$$H_{\mathbb{K}[G*\overline{\KK}_m]}(t)=H_{\mathbb{K}[G]}(t)+\frac{1}{(1-t)^m}-1.$$
\end{cor}

\begin{cor} Let $G$ be a simple graph and $\mathcal{S}_m$ be the star graph on $m+1$ vertices. Then
$$H_{\mathbb{K}[G*\mathcal{S}_m]}(t)=H_{\mathbb{K}[G]}(t)+\frac{1}{(1-t)^m}+\frac{1}{1-t}-2.$$
\end{cor}

\section{Basic setup}

For the convenience of the reader we include in this section the standard terminology and the basic facts which we will use throughout the paper.

A {\it simplicial complex} $\Delta$ over a set of vertices $V=\{x_1,\ldots,x_n\}$ is a subset of the powerset
of $V$ with that property that, whenever $F\in \Delta$ and $G\se F$, then $G\in\Delta$. The
elements of $\Delta$ are called {\it faces} and the {\it dimension} of a face is $\tn{dim}(F)=|F|-1$, where $|F|$ is the cardinality of $F$. Faces with dimension $0$ are called {\it vertices} and those with dimension $1$ are {\it edges}. A maximal face of $\Delta$ with respect to inclusion is called a {\it facet} of $\Delta$ and the {\it dimension} of $\Delta$, $\tn{dim}(\Delta)$, is the maximum dimension of its faces. If $\Delta$ is the  simplicial complex with the facets $F_1, . . . ,F_t$, then we write  $\Delta =\langle F_1, . . . ,F_t \rangle$.

Let $\Delta$ and $\Delta'$ be two simplicial complexes  with vertex sets $V$ and $V'$, respectively. The {\it union} $\Delta\cup\Delta'$ defines as the simplicial complex with the vertex set $V\cup V'$ and $F$ is a face of $\Delta\cup\Delta'$ if and only if $F$ is a face of $\Delta$ or $\Delta'$. If $V\cap V'=\emptyset$, then the {\it join} $\Delta * \Delta'$ is the simplicial complex on the vertex set $V\cup V'$ with faces $F\cup F'$ where $F\in \Delta$ and $F'\in \Delta'$. The {\it cone} of $\Delta$, denoted by $\tn{cone}(\Delta)$, is the join of a point $\{w\}$ with $\Delta$, that is, $\tn{cone}(\Delta)=\Delta *\{w\}$.

If $F\in \Delta$, then we define $x_F=\prod _{x_i\in F} x_i\in R=\mathbb{K}[x_1,\ldots,x_n]$ for some field $\mathbb{K}$. The {\it Stanley-Reisner ideal} of $\Delta$, denoted by $I_\Delta$ is
$$I_\Delta=(x_F~|~F\notin \Delta),$$
and the {\it Stanley-Reisner ring} of $\Delta$ is $\mathbb{K}[\Delta]=R/I_\Delta$.
For information on the basic theory of Stanley-Reisner ideals we refer the reader to \cite{BH} and \cite{S3}. In particular, in Section 4, we will make use of the fact that the $h$-vector of $\mathbb{K}[\Delta]$ is the $h$-vector of $\Delta$, i.e. if $h(\Delta)=(h_0,h_1,\ldots,h_{d+1})$  is the $h$-vector of $\Delta$, and $H_{\mathbb{K}[\Delta]}(t)$ is the Hilbert series of $\mathbb{K}[\Delta]$, then
$$H_{\mathbb{K}[\Delta]}(t)=\frac{\sum_{i=0}^{d+1}h_it^i}{(1-t)^{d+1}},$$
where $d=\tn{dim}(\Delta)$ (see for example, Section 5.1 in \cite{BH} for more details).

Let $M$ be an arbitrary graded $R$-module, and
let$$0\rightarrow\bigoplus_j{R(-j)^{\beta_{t,j}(M)}}\rightarrow\bigoplus_j{R(-j)^{\beta_{t-1,j}(M)}}\rightarrow
\cdots\rightarrow\bigoplus_j{R(-j)^{\beta_{0,j}(M)}}\rightarrow
M\rightarrow 0$$ be a minimal graded free resolution of $M$ over $R$,
where $R(-j)$ is a graded free $R$-module whose $n$th graded
component is given by $R_{n-j}$. The number $\beta_{i,j}(M)$ is
called the {\it $ij$th graded Betti number} of $M$ and equals the
number of generators of degree $j$ in the $i$th syzygy module. The
{\it Castelnuovo-Mumford regularity} of $M$ denoted by $\tn{reg}(M)$ is
defined by:
$$\tn{reg}(M) =\max\{j-i~|~\beta_{i,j}(M)\neq 0 \}.$$
Recall that the {\it projective dimension} of an $R$-module $M$,
denoted by $\tn{pd}(M)$, is the length of the minimal free resolution
of $M$, that is,
$$\tn{pd}(M) = \max\{i~|~\beta_{i,j}(M)\neq 0~\textnormal{for some}~j\}.$$

There is a strong connection between the topology of the simplicial
complex $\Delta$ and the structure of the free resolution of
$\mathbb{K}[\Delta]$. Let $\beta_{i,j}(\Delta)$ denotes the $\mathbb{N}$-graded Betti
numbers of the Stanley-Reisner ring $\mathbb{K}[\Delta]$. One of the most
well-known results is the Hochster's formula (\cite[Theorem
5.1]{Ho}).

\begin{thm} \label{Hoc}{\bf (Hochster's formula)} For $i>0$ the $\mathbb{N}$-graded Betti numbers $\beta_{i,j}$ of a
simplicial complex $\Delta$ are given by
$$\beta_{i,j}(\Delta)=\sum_{\substack{W\se V(\Delta)\\ |W|=j}}\textnormal{dim}_\mathbb{K}
\widetilde{H}_{j-i-1}(\Delta|_W;\mathbb{K}).$$
\end{thm}

Let $G$ be a simple graph. A subset $F$ of $V(G)$ is called an {\it independent set} of $G$ if  any subsets of $F$ with cardinality two do not belong to $E(G)$. The family of all independent sets of $G$ is a simplicial complex on the vertex set $V(G)$,
which is called the {\it independence complex} of $G$ and is denoted by  $\Delta_G$. We will use $\beta_{i,j}(G)$ for the $\mathbb{N}$-graded Betti numbers of $\Delta_G$. Since $\mathbb{K}[\Delta_G]=R/I(G)$, we have $\beta_{i,j}(G)=\beta_{i,j}(R/I(G))$.

The {\it cover ideal} of $G$, denoted by
$I(G)^\vee$, is defined to be the square-free monomial ideal
$$I(G)^\vee=(x_F~|~ F \mbox { is a (minimal) vertex cover of }G~),$$
where $x_F=\prod _{x_i\in F} x_i.$

We require the following result of N.Terai which is in \cite{T}.

\begin{lem} \label{lem:dual}
For a graph $G$ we have $\tn{reg}(R/I(G))=\tn{pd}(I(G)^\vee)$.
\end{lem}

The {\it Complement} of a graph $G$ is the graph $\overline{G}$ with the vertex set $V(G)$ and edges all the pairs $\{x_i,x_j\}$
such that $i\neq j$ and $\{x_i,x_j\}\notin E(G)$. Also for $W\se V(G)$ we use $G\setminus W$ for the subgraph of $G$ with the vertex set $V(G)\setminus W$ whose edge set is $\{ \{x,y \} \in E(G) ~|~  \{x,y \} \cap W = \emptyset\}$. We call a graph {\it discrete} if it has no edges.

Throughout this paper we use $\mathcal{C}_n$ for the cycle graph on $n$ vertices, i.e. $V(\mathcal{C}_n)=\{x_1,...,x_n\}$ and
$E(\mathcal{C}_n)=\{\{x_1,x_2\},...,\{x_{n-1},x_n\},\{x_n,x_1\}\}$, and
$\mathcal{W}_n$ for the wheel graph on $n+1$ vertices, i.e.
$V(\mathcal{W}_n)=\{x_1,...,x_{n+1}\}$ and
$E(\mathcal{W}_n)=\{\{x_1,x_2\},...,\{x_{n-1}\\,x_n\},\{x_n,x_1\},\{x_1,x_{n+1}\},...,\{x_n,x_{n+1}\}\}$. Also we use $\KK_n$ for the complete graph over $n$ vertices, and $\KK_{m,n}$ for the complete bipartite graph with vertex partition $V\cup V'$ where $|V|=m$ and $|V'|=n$.

\section{Betti numbers and regularity}

In this section we will make use of Whieldon's result on Betti numbers of product of graphs to reprove some results on Betti numbers of some families of graphs and also we will bring some results on the (Castelnuovo-Mumford) regularity of graph ideals. We begin this section with the next lemma which can be proved using Hochster's formula (Theorem \ref{Hoc}) with a similar argument as in the proof of \cite[Lemma 5.4]{W}.

\begin{lem} \label{lem:un}Let $\Delta_1$ and $\Delta_2$ be two simplicial complexes with disjoint vertex sets having $m$ and $n$ vertices, respectively. Also Let $\Delta=\Delta_1\cup\Delta_2$. Then the $\mathbb{N}$-graded Betti numbers $\beta_{i,d}(\Delta)$ can be expressed as
\begin{gather*}
\begin{cases}
\sum_{j=0}^{d-2}\{\binom{n}{j}\beta_{i-j,d-j}(\Delta_1)+\binom{m}{j}\beta_{i-j,d-j}(\Delta_2)\}&{~\tn{ if }~ d\neq i+1}\\
\\
\sum_{j=0}^{d-2}\{\binom{n}{j}\beta_{i-j,d-j}(\Delta_1)+\binom{m}{j}\beta_{i-j,d-j}(\Delta_2)\}+
\sum_{j=1}^{d-1}\binom{m}{j}\binom{n}{d-j}&{~\tn{ if }~ d= i+1}.
\end{cases}
\end{gather*}
\end{lem}

\begin{lem} \label{lem:disj} Let $G$ and $H$ be two simple graphs whose vertex sets are disjoint. Then $\Delta_{G*H}=\Delta_G\cup\Delta_H$ is the disjoint union of two simplicial complexes.
\end{lem}

\begin{proof} First suppose $F$ is a facet of $\Delta_{G*H}$. Then $F$ is an independent set of $G*H$. Since $G*H$ contains every edge of the form $\{x,y\}$ that $x\in V(G)$ and $y\in V(H)$, we get that $F\se V(G)$ or $F\se V(H)$. So $F\in \Delta_G$ or $F\in \Delta_H$. Conversely, if $F$ is a facet of $\Delta_G$, then $F$ is an independent set in $G$ and hence in $G*H$. Therefore $F\in\Delta_{G*H}$.
\end{proof}

\begin{rem} Since a Cohen-Macaulay complex of positive dimension is connected, we get the result that $G*H$ is Cohen-Macaulay if and only if $G$ and $H$ are complete graphs.
\end{rem}

Let $\beta_{i,j}(G*H)$ denotes the $\mathbb{N}$-graded Betti numbers of $\Delta_{G*H}$. Then we have the following translation of Lemma \ref{lem:un} to edge ideals which is Lemma 5.4 in \cite{W}.

\begin{cor} \label{cor:prod} Let $G$ and $H$ be two simple graphs with disjoint vertex sets having $m$ and $n$ vertices, respectively. Then the $\mathbb{N}$-graded Betti numbers $\beta_{i,d}(G*H)$ may be expressed as
\begin{gather*}
\begin{cases}
\sum_{j=0}^{d-2}\{\binom{n}{j}\beta_{i-j,d-j}(G)+\binom{m}{j}\beta_{i-j,d-j}(H)\}&{~\tn{ if }~ d\neq i+1}\\
\\
\sum_{j=0}^{d-2}\{\binom{n}{j}\beta_{i-j,d-j}(G)+\binom{m}{j}\beta_{i-j,d-j}(H)\}+
\sum_{j=1}^{d-1}\binom{m}{j}\binom{n}{d-j}&{~\tn{ if }~ d= i+1}.
\end{cases}
\end{gather*}
\end{cor}

Now we bring some classical results using the above formula. The first one is Theorem 5.2.4 in \cite{J}.

\begin{cor} The $\mathbb{N}$-graded Betti numbers of the complete bipartite graph $\KK_{m,n}$ may be expressed as
\begin{gather*}
\beta_{i,d}(\KK_{m,n})=
\begin{cases}
0 &{~\tn{ if }~ d\neq i+1}\\
\\
\sum_{j=1}^{i}\binom{m}{j}\binom{n}{i-j+1}&{~\tn{ if }~ d= i+1}.
\end{cases}
\end{gather*}
\end{cor}

\begin{proof} It is enough to notice that $\KK_{m,n}=\overline{\KK}_m*\overline{\KK}_n$ and that all graded Betti numbers of $\overline{\KK}_t$ are zero since it is a discrete graph.
\end{proof}

\begin{cor} \label{cor:K1} Let $H$ be a simple graph with $m+1$ vertices and let $x\in V(H)$ be adjacent to all other vertices of $H$. Then
\begin{gather*}
\beta_{i,d}(H)=
\begin{cases}
\beta_{i,d}(G)+\beta_{i-1,d-1}(G)&{~\tn{ if }~ d\neq i+1}\\
\\
\beta_{i,d}(G)+\beta_{i-1,d-1}(G)+\binom{m}{i}&{~\tn{ if }~ d=i+1},
\end{cases}
\end{gather*}
where $G=H\setminus\{x\}$.
\end{cor}

\begin{proof} This is an immediate consequence of Corollary \ref{cor:prod} together with $H\simeq G*\KK_1$. It is enough to observe that $\binom{1}{j}\neq 0$ provided $j\in\{0,1\}$, and that $\sum_{j=1}^{i}\binom{m}{j}\binom{1}{i-j+1}=\binom{m}{i}$.
\end{proof}

As an special case of Corollary \ref{cor:K1} we have the next result which is the content of \cite[Theorem 5.1]{EFMM}.

\begin{cor} Let $\WW_m$ denotes the wheel graph with $m+1$ vertices and let $\CC_m$ be the cycle graph with $m$ vertices. Then
\begin{gather*}
\beta_{i,d}(\WW_m)=
\begin{cases}
\beta_{i,d}(\CC_m)+\beta_{i-1,d-1}(\CC_m)&{~\tn{ if }~ d\neq i+1}\\
\\
\beta_{i,d}(\CC_m)+\beta_{i-1,d-1}(\CC_m)+\binom{m}{i}&{~\tn{ if }~ d=i+1}.
\end{cases}
\end{gather*}
\end{cor}

The star graph $\mathcal{S}_m$ is a graph over $m+1$ vertices with one vertex having vertex degree $m$ and the other $m$ vertices having vertex degree $1$. Since $\mathcal{S}_m=\overline{\KK}_m*\KK_1$, the following result is straightforward which is the content of \cite[Theorem 5.4.11]{J}.

\begin{cor} \label{star} Let $\mathcal{S}_m$ denotes the star graph with $m+1$ vertices. Then
\begin{gather*}
\beta_{i,d}(\mathcal{S}_m)=
\begin{cases}
0&{~\tn{ if }~ d\neq i+1}\\
\\
\binom{m}{i}&{~\tn{ if }~ d=i+1}.
\end{cases}
\end{gather*}
\end{cor}

\begin{cor} Let $G$ and be a simple graphs with $m$ vertices and let $\mathcal{S}_n$ denotes the star graph over $n+1$ vertices. Then the $\mathbb{N}$-graded Betti numbers $\beta_{i,d}(G*\mathcal{S}_n)$ may be expressed as
\begin{gather*}
\begin{cases}
\sum_{j=0}^{d-2}\binom{n+1}{j}\beta_{i-j,d-j}(G)&{~\tn{ if }~ d\neq i+1}\\
\\
\sum_{j=0}^{d-2}\binom{n+1}{j}\beta_{i-j,d-j}(G)+\binom{m+n+1}{i+1}+\binom{m+n}{i}-\binom{m+1}{i+1}-\binom{n+1}{i+!}&{~\tn{ if }~ d= i+1}.
\end{cases}
\end{gather*}
\end{cor}

\begin{proof} In view of Corollary \ref{star}, it suffices to prove the formula in the case where $d=i+1$. Since $\beta_{i-j,i-j+1}(\mathcal{S}_n)=\binom{n}{i-j}$, using Corollary \ref{cor:prod} we get
\begin{equation*}\begin{split}
\beta_{i,d}(G*\mathcal{S}_n)=\sum_{j=0}^{d-2}\binom{n+1}{j}\beta_{i-j,d-j}(G)+\sum_{j=0}^{d-2}\binom{m}{j}\binom{n}{i-j}
+\sum_{j=1}^{i}\binom{m}{j}\binom{n+1}{i-j+1}.
\end{split}\end{equation*}
To complete the proof it is enough to notice that
$$\sum_{j=0}^{i-1}\binom{m}{j}\binom{n}{i-j}=\binom{m+n}{i}-\binom{m}{i},$$
and
$$\sum_{j=1}^{i}\binom{m}{j}\binom{n+1}{i-j+1}=\binom{m+n+1}{i+1}-\binom{n+1}{i+1}-\binom{m}{i+1}.$$

\end{proof}

In the next result we determine the projective dimension of product of graphs.

\begin{cor} \label{PD} Let $G$ and $H$ be two simple graphs having $m$ and $n$ vertices, respectively. Then $\tn{pd}(R/I(G*H))=m+n-1$.
\end{cor}

\begin{proof} One has
$$\beta_{m+n-1,m+n}(G*H)\geq \sum_{j=1}^{m+n-1}\binom{m}{j}\binom{n}{m+n-j}\geq \binom{m}{m}\binom{n}{n}=1.$$
On the other hand since $\beta_{i,j}(L)=0$ for any simple graph $L$ and $i=|V(L)|$, we get that $\beta_{m+n}(G*H)=0$. Therefore $\tn{pd}(R/I(G*H))=m+n-1$.
\end{proof}

Now we are going to determine the Castelnuovo-Mumford regularity of product of graphs.

\begin{cor} \label{LR} Let $G$ and $H$ be two simple graphs with disjoint vertex sets. Then $G*H$ has linear resolution if and only if $G$ and $H$ have.
\end{cor}

\begin{proof} First note that if $\beta_{i,j}(G)\neq 0$ or $\beta_{i,j}(H)\neq 0$ for some $i$ and some $j$, then $\beta_{i,j}(G*H)\neq 0$. Now $G*H$ has linear resolution if and only if  $\beta_{i,d}(G*H)=0$ for all $d\neq i+1$. This equality holds if and only if  $\beta_{i,d}(G)=\beta_{i,d}(H)= 0$  for all $d\neq i+1$, i.e. $G$ and $H$ have linear resolutions.
\end{proof}

\begin{prop} \label{reg} Let $G$ and $H$ be two simple graphs with disjoint vertex sets. Then
$$\tn{reg}(R/I(G*H))=\tn{max}\{\tn{reg}(R/I(G)),\tn{reg}(R/I(H))\}.$$
\end{prop}

\begin{proof} First assume $\tn{max}\{\tn{reg}(R/I(G)),\tn{reg}(R/I(H))\}=s$. In view of Corollary \ref{LR} we may assume $s>1$. Now suppose $\tn{reg}(R/I(G))=s$. It follows that $\beta_{i,i+s}(G)\neq 0$ for some $i$, which implies that $\beta_{i,i+s}(G*H)\neq 0$. Therefore $\tn{reg}(R/I(G*H))\geq s$. Also note that if there exists $d>s$ such that $\beta_{i,i+d}(G*H)\neq 0$, then $\beta_{i-j,i+d-j}(G)\neq 0$ or $\beta_{i-j,i+d-j}(H)\neq 0$ for some $0\leq j \leq i+d-2$, a contradiction.
\end{proof}

\begin{rem} Proposition \ref{reg} gives a procedure to construct a family of graphs with the property that the regularity of its elements is a given integer. Indeed suppose $s>0$ is an integer and assume $G$ is a simple graph with $\tn{reg}(R/I(G))=s$ (for example, choose $G$ as the cycle graph over $3s$ vertices (see \cite{J})). Then for any graph $H$ with $\tn{reg}(R/I(H))\leq s$ one has $\tn{reg}(R/I(G*H))=s$.
\end{rem}

It is easy to see that any set of pairwise $3$-disjoint edges of $G$ is also a set of pairwise $3$-disjoint edges of $G*H$ and hence $\tn{max}\{a(G),a(H)\}\leq a(G*H)$. Conversely, if $A$ is a set of pairwise $3$-disjoint edges of $G*H$, then $A\se E(G)$  or $A\se E(H)$, i.e. $A$ is a set of pairwise $3$-disjoint edges of $G$ or $H$. Therefore we have proved the next lemma.

\begin{lem} \label{a} Let $G$ and $H$ be two simple graphs with disjoint vertex sets. Then
$$a(G*H)=\tn{max}\{a(G),a(H)\}.$$
\end{lem}

\begin{cor} \label{eqp} Let $G$ be a simple graph with $\tn{reg}(R/I(G))=a(G)$. Then for any graph $H$ with $\tn{reg}(R/I(H))\leq\tn{reg}(R/I(G))$ one has
$$\tn{reg}(R/I(G*H))=a(G*H).$$
\end{cor}

\begin{proof} It follows from Theorem \ref{regT} that $a(G)=\tn{reg}(R/I(G))\geq \tn{reg}(R/I(H))\geq a(H)$. Now using Proposition \ref{reg} and Lemma \ref{a} we get
$$\tn{reg}(R/I(G*H))=\tn{reg}(R/I(G))=a(G)=a(G*H).$$
\end{proof}

Let $\mathcal{A}$ be the set of all graphs whose regularity equal the maximum number of pairwise 3-disjoint edges, i.e.
$$\mathcal{A}=\{~G~ ~ ~|~ ~ ~ G ~\tn{is a simple graph with } \tn{reg}(R/I(G))=a(G)\}.$$
There has been several attempts to determine elements of $\mathcal{A}$.
Zheng \cite{Z} proved that trees belong to $\mathcal{A}$.
Francisco, H\`a and Van Tuyl \cite{FHVT} proved $\mathcal{A}$ contains Cohen-Macaulay bipartite graphs. Van Tuyl \cite{VT} generalized this to the family of sequentially Cohen-Macaulay bipartite graphs. Note that a tree is a sequentially Cohen-Macaulay bipartite graph (see \cite{F}). Kummini \cite{Ku} proved that $\mathcal{A}$ contains unmixed bipartite graphs. Also the authors in \cite{MMCRTY} generalized Kummini's result to the class of very well-covered graphs. In addition, $\mathcal{A}$ contains all cycles $\CC_m$ that $m\equiv 0$ and $1$ mod $3$ (see \cite{J}). Using Corollary \ref{eqp} we have the following.

\begin{prop} \label{A} Let $\mathcal{A}$ be the set of all graphs $G$ with the property $\tn{reg}(R/I(G))=a(G)$. Then $\mathcal{A}$ is closed under product of graphs, i.e. for any $G,H\in \mathcal{A}$ one has
$$\tn{reg}(R/I(G*H))=a(G*H).$$
\end{prop}

\begin{rem} Using Proposition \ref{A} one can construct graphs that do not belong to the above mentioned families but their regularity equal the maximum number of pairwise 3-disjoint edges. For example consider $\CC_3*\CC_4$. One can easily see that this graph is not bipartite, not unmixed, not sequentially Cohen-Macaulay, and not very well-covered, but its regularity equals the maximum number of pairwise 3-disjoint edges which is 1.
\end{rem}

Let $G$ and $H$ be two simple graphs with disjoint vertex sets $X$ and $Y$, respectively. It is easy to see that minimal vertex covers of $G*H$ are of the following forms:
\begin{itemize}
\item[(1)] $A\cup Y$, where $A$ is a minimal vertex cover of $G$,
\item[(2)] $X\cup B$, where $B$ is a minimal vertex cover of $H$.
\end{itemize}
It follows that the cover ideal $I(G*H)^\vee$ of $G*H$ can be written as
$$I(G*H)^\vee=XI(H)^\vee +YI(G)^\vee.$$
Therefore we have the following.

\begin{cor} Let $G$ and $H$ be two simple graphs with disjoint vertex sets $X$ and $Y$ having $m$ and $n$ vertices, respectively. Then
\begin{itemize}
\item[(i)] $\tn{pd}(XI(H)^\vee +YI(G)^\vee) =\tn{max}\{\tn{pd}(I(G)^\vee),\tn{pd}(I(H)^\vee)\},$
\item[(ii)] $\tn{reg}(XI(H)^\vee +YI(G)^\vee)=m+n-1.$
\end{itemize}
\end{cor}

\begin{proof} (i) Using Lemma \ref{lem:dual} and Proposition \ref{reg} one has
\begin{equation*}\begin{split}
\tn{pd}(I(G*H)^\vee)&=\tn{reg}(R/I(G*H))\\
&=\tn{max}\{\tn{reg}(R/I(G)),\tn{reg}(R/I(H))\}\\
&=\tn{max}\{\tn{pd}(I(G)^\vee),\tn{pd}(I(H)^\vee)\}.
\end{split}\end{equation*}
(ii) It follows from Lemma \ref{lem:dual} that for any graph $G$ we have $\tn{reg}(I(G)^\vee)=\tn{pd}(R/I(G))$, which together with Corollary \ref{PD} completes the proof.
\end{proof}

\section{$h$-vectors and Hilbert series}
In this section we first compute $h$-vector of disjoint union of two simplicial complexes and then using that we will provide an explicit formula to compute Hilbert series of a simplical complex in term of Hilbert series of its connected components. We begin this section with the next result on the $h$-vector of disjoint union of two simplicial complexes with the same dimension and then we generalize our result to the case where the dimensions are not equal.

\begin{lem} \label{eqdim} Let $\Delta$ and $\Delta'$be two simplicial complexes with disjoint vertex sets, and let $\tn{dim}(\Delta)=\tn{dim}(\Delta')=d$. Then
$$h_k(\Delta\cup\Delta')=h_k(\Delta)+h_k(\Delta')-(-1)^k\binom{d+1}{k},$$
for all $0\leq k \leq d+1$.
\end{lem}

\begin{proof} Note that
\begin{gather*}
f_i(\Delta\cup\Delta')=
\begin{cases}
1 &{~\tn{ if }~ i=-1}\\
\\
f_i(\Delta)+f_i(\Delta') &{~\tn{ if }~ i\neq -1}.
\end{cases}
\end{gather*}
Therefore we have
\begin{equation*}\begin{split}
h_k(\Delta\cup\Delta')&=\sum_{i=0}^k (-1)^{k-i}\binom{d+1-i}{k-i}f_{i-1}(\Delta\cup\Delta')\\
&=\sum_{i=0}^k (-1)^{k-i}\binom{d+1-i}{k-i}f_{i-1}(\Delta)\\
&\qquad +\sum_{i=0}^k (-1)^{k-i}\binom{d+1-i}{k-i}f_{i-1}(\Delta')-(-1)^k\binom{d+1}{k}\\
&=h_k(\Delta)+h_k(\Delta')-(-1)^k\binom{d+1}{k}.
\end{split}\end{equation*}
\end{proof}

To generalize Lemma \ref{eqdim} to the case where the dimensions are not equal, we need the next result.

\begin{lem} \label{cone} Let $\Delta'$ be a simplicial complex with dimension $d-1$. Then
$$h(\tn{cone}(\Delta'))=(h(\Delta'),0).$$
\end{lem}

\begin{proof} Assume $f(\Delta')=(g_0,g_1,\ldots,g_{d-1})$. It is easy to see that
$$f(\tn{cone}(\Delta'))=(g_0+g_{-1},g_1+g_0,\ldots,g_{d-1}+g_{d-2},g_d+g_{d-1}),$$
where $g_d=0$. Hence for $k\leq d$ one has
\begin{equation*}\begin{split}
h_k(\tn{cone}(\Delta'))&=\sum_{i=0}^k (-1)^{k-i}\binom{d+1-i}{k-i}f_{i-1}(\tn{cone}(\Delta'))\\
&=\sum_{i=0}^k (-1)^{k-i}\binom{d+1-i}{k-i}g_{i-1}+\sum_{i=0}^k (-1)^{k-i}\binom{d+1-i}{k-i}g_{i-2}.
\end{split}\end{equation*}
Using identity $\binom{d+1-i}{k-i}=\binom{d-i}{k-i}+\binom{d-i}{k-i-1}$, the first sum above can be written as
$\sum_{i=0}^k (-1)^{k-i}\binom{d-i}{k-i}g_{i-1}+\sum_{i=0}^k (-1)^{k-i}\binom{d-i}{k-i-1}g_{i-1}.$
Note that if $i=k$, then $\binom{d-i}{k-i-1}=0$ and hence
$$\sum_{i=0}^k (-1)^{k-i}\binom{d+1-i}{k-i}g_{i-1}=h_k(\Delta')-\sum_{i=0}^{k-1} (-1)^{k-i-1}\binom{d-i}{k-i-1}g_{i-1}.$$
On the other hand, since $g_{-2}=0$, by setting $j=i-1$ in $\sum_{i=1}^k (-1)^{k-i}\binom{d+1-i}{k-i}g_{i-2}$, this term can be written as $\sum_{j=0}^{k-1} (-1)^{k-j-1}\binom{d-j}{k-j-1}g_{j-1}$ which implies $h_k(\tn{cone}(\Delta'))=h_k(\Delta')$ for all $k\leq d$.\\
To complete the proof note that
\begin{equation*}\begin{split}
h_{d+1}(\tn{cone}(\Delta'))&=\sum_{i=0}^{d+1}(-1)^{d+1-i}\binom{d+1-i}{d+1-i}g_{i-1}
+\sum_{i=0}^{d+1}(-1)^{d+1-i}\binom{d+1-i}{d+1-i}g_{i-2}\\
&=\sum_{i=0}^{d}(-1)^{d+1-i}g_{i-1}+\sum_{i=1}^{d+1}(-1)^{d+1-i}g_{i-2}\\
&=\sum_{i=0}^{d}(-1)^{d+1-i}g_{i-1}+\sum_{j=0}^{d}(-1)^{d-j}g_{j-1}=0.
\end{split}\end{equation*}
\end{proof}

For $n\in \mathbb{N}$, let $[n]$ denotes the full simplex over $n$ vertices. Applying this notation one has $\tn{cone}(\Delta')=\Delta'*[1]$, where $(*)$ denotes the join of simplicial complexes. Therefore in Lemma \ref{cone} we have proved if $h(\Delta')=(h_0,h_1,\ldots,h_d)$, then $h(\Delta'*[1])=(h_0,h_1,\ldots,h_d,0)$. Using induction on $n$ the next corollary is straightforward.

\begin{cor} \label{*[n]} Let $\Delta'$ be a simplicial complex with dimension $d-1$ whose $h$-vector is $h(\Delta')=(h_0,h_1,\ldots,h_d)$. If $[n]$ denotes the full simplex over $n$ vertices, then
\begin{gather*}
h(\Delta'*[n])=(h_0,h_1,\ldots,h_d,\overbrace{0,0,\ldots,0}^{n-times}).
\end{gather*}
\end{cor}

\begin{rem} It follows from Corollary \ref{*[n]} that if $H_{\mathbb{K}[\Delta']}(t)=Q(t)/(1-t)^d$ is the Hilbert series of $\mathbb{K}[\Delta']$, where $\tn{dim}(\Delta')=d-1$, then for all $n\in \mathbb{N}$, the Hilbert series of $\mathbb{K}[\Delta'*[n]]$ is $$H_{\mathbb{K}[\Delta'*[n]]}(t)=\frac{Q(t)}{(1-t)^{d+n}}=\frac{H_{\mathbb{K}[\Delta']}(t)}{(1-t)^{n}}.$$
\end{rem}

To investigate the $h$-vector of disjoint union of complexes with different dimensions we need the following lemma.

\begin{lem} \label{sum} Let $1\leq p\leq n$ be integers. Then
$$\sum_{t=1}^p(-1)^{p-t+1}\binom{n}{t}\binom{n-t}{p-t}=(-1)^p\binom{n}{p}.$$
\end{lem}

\begin{proof} We proceed by induction on $n$. If $n=1$, there is nothing to prove. So assume $n>1$ and the assertion is valid for $n$. We prove the assertion for $n+1$. Since
$$\binom{n+1}{t}\binom{n+1-t}{p-t}=\frac{n+1}{n-p+1}\binom{n}{t}\binom{n-t}{p-t},$$
we have
\begin{equation*}\begin{split}
\sum_{t=1}^p (-1)^{p-t+1}\binom{n+1}{t}\binom{n+1-t}{p-t}&=\sum_{t=1}^p(-1)^{p-t+1}\frac{n+1}{n-p+1}\binom{n}{t}\binom{n-t}{p-t}\\
&=\frac{n+1}{n-p+1}(-1)^p\binom{n}{p}=(-1)^p\binom{n+1}{p}.
\end{split}\end{equation*}
\end{proof}

Now we are ready to compute $h$-vector of $\Delta\cup\Delta'$ in terms of $h$-vectors of $\Delta$ and $\Delta'$, where $\tn{dim}(\Delta)\neq \tn{dim}(\Delta')$. Recall that by definition, for any simplicial complex $\Delta$ of dimension $d$ we have $h_k(\Delta)=0$ for $k<0$ and also for $k>d+1$.

\begin{prop} \label{neqdim} Let $\Delta$ and $\Delta'$ be two simplicial complexes with disjoint vertex sets. Also let $\tn{dim}(\Delta')\leq \tn{dim}(\Delta)$, and assume $n=\tn{dim}(\Delta)-\tn{dim}(\Delta')$. Then
$$h_k(\Delta\cup\Delta')=h_k(\Delta)+\sum_{p=0}^{n}(-1)^p\binom{n}{p}h_{k-p}(\Delta')-(-1)^k\binom{d+1}{k},$$
for all $0\leq k \leq d+1$, where $d=\tn{dim}(\Delta)$.
\end{prop}

\begin{proof} Suppose $\Delta''=\Delta'*[n]$. It follows that $\tn{dim}(\Delta'')=d$, and hence by Lemma \ref{eqdim} and Corollary \ref{*[n]} one has
\begin{equation}\begin{split}
h_k(\Delta\cup\Delta'')&=h_k(\Delta)+h_k(\Delta'')-(-1)^k\binom{d+1}{k}\\
&=h_k(\Delta)+h_k(\Delta')-(-1)^k\binom{d+1}{k}.
\end{split}\end{equation}
Now assume $\tn{dim}(\Delta')=d'$ and $f(\Delta')=(g_0,g_1,\ldots,g_{d'})$. Applying induction on $n$ we get $f_i(\Delta'*[n])=\sum_{t=0}^n\binom{n}{t}g_{i-t}$ and hence $f_i(\Delta\cup\Delta'')=f_i(\Delta\cup\Delta')+\sum_{t=1}^n\binom{n}{t}g_{i-t}$. Therefore
\begin{equation}\begin{split}
h_k(\Delta\cup\Delta'')&=\sum_{i=0}^k (-1)^{k-i}\binom{d+1-i}{k-i}f_{i-1}(\Delta\cup\Delta'')\\
&=\sum_{i=0}^k (-1)^{k-i}\binom{d+1-i}{k-i}f_{i-1}(\Delta\cup\Delta')\\
&\qquad +\sum_{i=0}^k (-1)^{k-i}\binom{d+1-i}{k-i}\sum_{t=1}^n\binom{n}{t}g_{i-t-1}\\
&=h_k(\Delta\cup\Delta')+\sum_{t=1}^n\binom{n}{t}\sum_{i=0}^k (-1)^{k-i}\binom{d+1-i}{k-i}g_{i-t-1}.
\end{split}\end{equation}
Suppose $1\leq t \leq n$ is fixed. Since  $g_{i-t-1}=0$ for all $i<t$, one has
\begin{equation*}\begin{split}
\sum_{i=0}^k (-1)^{k-i}\binom{d+1-i}{k-i}g_{i-t-1}&=\sum_{i=t}^k (-1)^{k-i}\binom{d+1-i}{k-i}g_{i-t-1}\\
&=\sum_{j=0}^{k-t} (-1)^{(k-t)-j}\binom{(d-t+1)-j}{(k-t)-j}g_{j-1}.
\end{split}\end{equation*}
Now we apply the identity $\binom{\ell}{m}=\sum_{u=0}^{s}\binom{s}{u}\binom{\ell-s}{m-u}$ which is valid for all $0\leq s\leq \ell$ to obtain
$$\binom{(d-t+1)-j}{(k-t)-j}=\sum_{u=0}^{n-t}\binom{n-t}{u}\binom{(d-n+1)-j}{(k-t-u)-j}.$$
(Note that $n-t\leq d-t+1-j$ if and only if $d'=d-n\geq j-1$ which is true because $j-1\leq k-t-1\leq d'$ if and only if $k-1\leq d'+t\leq d'+n=d$ which is obvious.)\\
It follows that
\begin{equation*}\begin{split}
\sum_{i=0}^k (-1)^{k-i}&\binom{d+1-i}{k-i}g_{i-t-1}\\
&=\sum_{j=0}^{k-t} (-1)^{(k-t)-j}\sum_{u=0}^{n-t}\binom{n-t}{u}\binom{(d-n+1)-j}{(k-t-u)-j}g_{j-1}\\
&=\sum_{u=0}^{n-t}\binom{n-t}{u}\sum_{j=0}^{k-t} (-1)^{(k-t)-j}\binom{(d-n+1)-j}{(k-t-u)-j}g_{j-1}\\
&=\sum_{u=0}^{n-t}(-1)^u\binom{n-t}{u}\sum_{j=0}^{k-t-u} (-1)^{(k-t-u)-j}\binom{(d'+1)-j}{(k-t-u)-j}g_{j-1}\\
&=\sum_{u=0}^{n-t}(-1)^u\binom{n-t}{u}h_{k-t-u}(\Delta').
\end{split}\end{equation*}
Replacing the above relation in Equation 4.2 we get that
\begin{equation*}\begin{split}
h_k(\Delta\cup\Delta'')&=h_k(\Delta\cup\Delta')+\sum_{t=1}^n\sum_{u=0}^{n-t}(-1)^u\binom{n}{t}\binom{n-t}{u}h_{k-t-u}(\Delta')\\
&=h_k(\Delta\cup\Delta')+\sum_{p=1}^nC_ph_{k-p}(\Delta'),
\end{split}\end{equation*}
where $C_p$ is the coefficient of $h_{k-p}(\Delta')$ in the above. Indeed, since $0\leq t \leq n$ and $0\leq u\leq n-t$, we have $0\leq p=t+u\leq n$. Henceforth a direct computation shows that $C_p=\sum_{t=1}^p(-1)^{p-t}\binom{n}{t}\binom{n-t}{p-t}$ which together with Lemma \ref{sum} implies that $C_p=(-1)^{p+1}\binom{n}{p}$. Hence
$$h_k(\Delta\cup\Delta'')=h_k(\Delta\cup\Delta')+\sum_{p=1}^n(-1)^{p+1}\binom{n}{p}h_{k-p}(\Delta').$$
Comparing this with Equation 4.1 shows that
\begin{equation*}\begin{split}
h_k(\Delta\cup\Delta')&=h_k(\Delta)+h_k(\Delta')+\sum_{p=1}^n(-1)^{p}\binom{n}{p}h_{k-p}(\Delta')-(-1)^k\binom{d+1}{k}\\
&=h_k(\Delta)+\sum_{p=0}^n(-1)^{p}\binom{n}{p}h_{k-p}(\Delta')-(-1)^k\binom{d+1}{k}.
\end{split}\end{equation*}
\end{proof}

\begin{rem} It is easy to see that in the case where $k<p<k-d'-1$ one has $h_{k-p}(\Delta')=0$. Therefore the above formula can be written as
$$h_k(\Delta\cup\Delta')=h_k(\Delta)+\sum_{p=\tn{max}\{0,k-d'-1\}}^{\tn{min}\{n,k\}}(-1)^{p}
\binom{n}{p}h_{k-p}(\Delta')-(-1)^k\binom{d+1}{k}.$$
\end{rem}

\begin{exam} Let $h(\Delta)=(1,2,-4,1,2,-1)$ and $h(\Delta')=(1,2,0)$. Then $\tn{dim}(\Delta)=4$, $\tn{dim}(\Delta')=1$, and $n=\tn{dim}(\Delta)-\tn{dim}(\Delta')=3$. Using Proposition \ref{neqdim} we obtain
\begin{equation*}\begin{split}
(1)\quad &h_1(\Delta\cup\Delta')=h_1(\Delta)+h_1(\Delta')-3h_0(\Delta')+\binom{5}{1}=6,\\
(2)\quad &h_2(\Delta\cup\Delta')=h_2(\Delta)+h_2(\Delta')-3h_1(\Delta')+3h_0(\Delta')-\binom{5}{2}=-17,\\
(3)\quad &h_3(\Delta\cup\Delta')=h_3(\Delta)+h_3(\Delta')-3h_2(\Delta')+3h_1(\Delta')-h_0(\Delta')+\binom{5}{3}=16,\\
(4)\quad &h_4(\Delta\cup\Delta')=h_4(\Delta)+h_4(\Delta')-3h_3(\Delta')+3h_2(\Delta')-h_1(\Delta')-\binom{5}{4}=-5,\\
(5)\quad &h_5(\Delta\cup\Delta')=h_5(\Delta)+h_5(\Delta')-3h_4(\Delta')+3h_3(\Delta')-h_2(\Delta')+\binom{5}{5}=0.\\
\end{split}\end{equation*}
Therefore $h(\Delta\cup\Delta')=(1,6,-17,16,-5,0)$. Note that $\Delta=\langle\{x,y,z,t,u\},\{t,u,v\},\\\{v,w\}\rangle$ and $\Delta'=\langle\{a,b\},\{b,c\},\{c,d\}\rangle$, so that  a direct computation shows that $h(\Delta\cup\Delta')=(1,6,-17,16,-5,0)$.
\end{exam}

The next proposition provides Hilbert series of disjoint union of two complexes in terms of Hilbert series of the original complexes.

\begin{prop} \label{Hilbert} Let $\Delta$ and $\Delta'$ be two simplicial complexes with disjoint vertex sets. Then
$$H_{\mathbb{K}[\Delta\cup\Delta']}(t)=H_{\mathbb{K}[\Delta]}(t)+H_{\mathbb{K}[\Delta']}(t)-1.$$
\end{prop}

\begin{proof} Assume $\tn{dim}(\Delta')\leq \tn{dim}(\Delta)$, and set $n=\tn{dim}(\Delta)-\tn{dim}(\Delta')$. It follows from Proposition \ref{neqdim} that
\begin{equation*}\begin{split}
H_{\mathbb{K}[\Delta\cup\Delta']}(t)&=\frac{\sum_{i=0}^{d+1}h_i(\Delta)t^i+\sum_{i=0}^{d+1}(\sum_{p=0}^n(-1)^{p}\binom{n}{p}h_{i-p}(\Delta'))t^i
-\sum_{i=0}^{d+1}(-1)^i\binom{d+1}{i}t^i}{(1-t)^{d+1}}\\
&=H_{\mathbb{K}[\Delta]}(t)+\frac{\sum_{i=0}^{d+1}(\sum_{p=0}^n(-1)^{p}\binom{n}{p}h_{i-p}(\Delta'))t^i}{(1-t)^{d+1}}-1.\\
\end{split}\end{equation*}
Since $h_{i-p}(\Delta')=0$ for $i<p$ and for $i>d'+p+1$, one has
\begin{equation*}\begin{split}
\sum_{i=0}^{d+1}(\sum_{p=0}^n(-1)^{p}\binom{n}{p}h_{i-p}(\Delta'))t^i
&=\sum_{p=0}^n(-1)^{p}\binom{n}{p}\sum_{i=0}^{d+1}h_{i-p}(\Delta')t^i\\
&=\sum_{p=0}^n(-1)^{p}\binom{n}{p}\sum_{i=p}^{d'+p+1}h_{i-p}(\Delta')t^i\\
&=\sum_{p=0}^n(-1)^{p}\binom{n}{p}\sum_{j=0}^{d'+1}h_{j}(\Delta')t^{p+j}\\
&=\sum_{p=0}^n(-1)^{p}\binom{n}{p}t^p\sum_{j=0}^{d'+1}h_{j}(\Delta')t^j\\
&=(1-t)^n\sum_{j=0}^{d'+1}h_{j}(\Delta')t^j.
\end{split}\end{equation*}
Therefore we get that
\begin{equation*}\begin{split}
\frac{\sum_{i=0}^{d+1}(\sum_{p=0}^n(-1)^{p}\binom{n}{p}h_{i-p}(\Delta'))t^i}{(1-t)^{d+1}}
&=\frac{(1-t)^n\sum_{j=0}^{d'+1}h_{j}(\Delta')t^j}{(1-t)^{d+1}}\\
&=\frac{\sum_{j=0}^{d'+1}h_{j}(\Delta')t^j}{(1-t)^{d'+1}}=H_{\mathbb{K}[\Delta']}(t),
\end{split}\end{equation*}
which completes the proof.
\end{proof}

Using induction on the number of connected components of a simplicial complex we have the following result.

\begin{cor} \label{Gen:Hilbert} Let $\Delta$ be a simplicial complex and let $\Delta_1,\ldots,\Delta_r$ be connected components of $\Delta$. Then
$$H_{\mathbb{K}[\Delta]}(t)=\sum_{j=1}^rH_{\mathbb{K}[\Delta_j]}(t)-(r-1).$$
\end{cor}

We now bring some results on $h$-vector and Hilbert series of product of graphs. Recall that for any two graphs $G$ and $F$ with disjoint vertex sets one has $\Delta_{G*F}=\Delta_G\cup \Delta_F$ (as a disjoint union). Let $h(G)$ denotes the $h$-vector of the independence complex $\Delta_G$. By specializing Proposition \ref{neqdim} to the independence complex we have the following.

\begin{cor} \label{graph} Let $G$ and $F$ be two simple graphs with disjoint vertex sets. Also let $\tn{dim}(\Delta_F)\leq \tn{dim}(\Delta_G)$ and assume $n=\tn{dim}(\Delta_G)-\tn{dim}(\Delta_F)$. Then
$$h_k(G*F)=h_k(G)+\sum_{p=0}^{n}(-1)^p\binom{n}{p}h_{k-p}(F)-(-1)^k\binom{d+1}{k},$$
for all $0\leq k \leq d+1$, where $d=\tn{dim}(\Delta_G)$.
\end{cor}

We have also the following translation of Corollary \ref{Gen:Hilbert} to Hilbert series of product of graphs.

\begin{cor} \label{Hilbert:graph} Let $G_1,\ldots,G_r$ be simple graphs with disjoint vertex sets. Then
$$H_{\mathbb{K}[G_1*\cdots *G_r]}(t)=\sum_{j=1}^rH_{\mathbb{K}[G_j]}(t)-(r-1).$$
\end{cor}

\begin{proof} It suffices to notice that
$$\Delta_{G_1*\cdots *G_r}=\bigcup_{j=1}^r\Delta_{G_j},$$
and apply Corollary \ref{Gen:Hilbert}.
\end{proof}

We now bring some results on $h$-vector and Hilbert series of product of special graphs. The first one is about product of a graph by a complete one.

\begin{cor} \label{comp} Let $G$ be a simple graph with $\tn{dim}(\Delta_G)=d$ and let $\KK_m$ denotes the complete graph with $m$ vertices. Then
$$h_k(G*\KK_m)=h_k(G)+(-1)^{k-1}\binom{d}{k-1}m,$$
and
$$H_{\mathbb{K}[G*\KK_m]}(t)=H_{\mathbb{K}[G]}(t)+m\frac{t}{1-t}.$$
\end{cor}

\begin{proof} First note that $\tn{dim}(\Delta_{\KK_m})=0$ and $h(\KK_m)=(1,m-1)$. Now Corollary \ref{graph} yields
\begin{equation*}\begin{split}
h_k(G*\KK_m)&=h_k(G)+\sum_{p=0}^{d}(-1)^p\binom{d}{p}h_{k-p}(\KK_m)-(-1)^k\binom{d+1}{k}\\
&=h_k(G)+\sum_{p=k-1}^{k}(-1)^p\binom{d}{p}h_{k-p}(\KK_m)-(-1)^k\binom{d+1}{k}\\
&=h_k(G)+(-1)^{k-1}\binom{d}{k-1}(m-1)+(-1)^k\binom{d}{k}-(-1)^k\binom{d+1}{k}\\
&=h_k(G)+(-1)^{k-1}\binom{d}{k-1}m.
\end{split}\end{equation*}
For the second assertion on Hilbert series, using Corollary \ref{Hilbert:graph} we get
$$H_{\mathbb{K}[G*\KK_m]}(t)=H_{\mathbb{K}[G]}(t)+\frac{1+(m-1)t}{1-t}-1=H_{\mathbb{K}[G]}(t)+m\frac{t}{1-t}.$$
\end{proof}

As an immediate consequence we have the next corollary.

\begin{cor} \label{HK1} Let $G$ be a simple graph and let $x\in V(G)$ be adjacent to all other vertices of $G$. Then
$$H_{\mathbb{K}[G]}(t)=H_{\mathbb{K}[G\setminus\{x\}]}(t)+\frac{t}{1-t}.$$
In particular, if $\WW_n$ and $\mathcal{S}_n$ be the wheel graph and star graph on $n+1$ vertices, respectively, then
$$H_{\mathbb{K}[\WW_n]}(t)=H_{\mathbb{K}[\CC_n]}(t)+\frac{t}{1-t},$$
and
$$H_{\mathbb{K}[\mathcal{S}_n]}(t)=\frac{1+t(1-t)^{n-1}}{(1-t)^{n}}.$$
\end{cor}

\begin{proof} We just need to prove the last assertion. Note that $\Delta_{\overline{\KK}_m}=[m]$ is the full simplex over $m$ vertices and hence the $f$-vector of $\overline{\KK}_m$ is $f(\overline{\KK}_m)=(\binom{m}{1},\binom{m}{2},\ldots,\binom{m}{m})$. Therefore Lemma \ref{sum} yields that the $h$-vector of $\overline{\KK}_m$ is $h(\overline{\KK}_m)=(1,\underbrace{0,0,\ldots,0}_{m-\tn{times}})$. Now we have
$$H_{\mathbb{K}[\mathcal{S}_n]}(t)=H_{\mathbb{K}[\overline{\KK}_n]}(t)+\frac{t}{1-t}=
\frac{1}{(1-t)^n}+\frac{t}{1-t}=\frac{1+t(1-t)^{n-1}}{(1-t)^{n}}.$$
\end{proof}

We now consider product of a graph by the complement of complete graphs.

\begin{cor} \label{ccom} Let $G$ be a simple graph with $\tn{dim}(\Delta_G)=d$ and let $\overline{\KK}_m$ denotes the complement of $\KK_m$. Then
\begin{gather*}
h_k(G*\overline{\KK}_m)
\begin{cases}
h_k(G)-(-1)^k\{\binom{d+1}{k}-\binom{d-m+1}{k}\}&{~\tn{ if }~ m\leq d+1}
\\
\sum_{p=0}^{m-d-1}(-1)^p\binom{m-d-1}{p}h_{k-p}(G)-(-1)^k\binom{m}{k}&{~\tn{ if }~ m> d+1},
\end{cases}
\end{gather*}
for all $k>0$.
In particular,
$$h_k(\KK_{m,n})=(-1)^{k+1}\biggl\{\binom{n}{k}-\binom{n-m}{k}\biggr\},$$
for all $k>0.$
\end{cor}

\begin{proof} First assume $m\leq d+1$. Since $h(\overline{\KK}_m)=(1,\underbrace{0,0,\ldots,0}_{m-\tn{times}})$, one has
\begin{equation*}\begin{split}
h_k(G*\overline{\KK}_m)&=h_k(G)+\sum_{p=0}^{d-m+1}(-1)^p\binom{d-m+1}{p}h_{k-p}(\overline{\KK}_m)-(-1)^k\binom{d+1}{k}\\
&=h_k(G)+(-1)^k\binom{d-m+1}{k}-(-1)^k\binom{d+1}{k}\\
&=h_k(G)-(-1)^k\biggl\{\binom{d+1}{k}-\binom{d-m+1}{k}\biggr\}.
\end{split}\end{equation*}
Now assume $m>d+1$ and $k>0$. We have
\begin{equation*}\begin{split}
h_k(\overline{\KK}_m*G)&=h_k(\overline{\KK}_m)+\sum_{p=0}^{m-d-1}(-1)^p\binom{m-d-1}{p}h_{k-p}(G)-(-1)^k\binom{m}{k}\\
&=\sum_{p=0}^{m-d-1}(-1)^p\binom{m-d-1}{p}h_{k-p}(G)-(-1)^k\binom{m}{k}.
\end{split}\end{equation*}
The last assertion follows from $\KK_{m,n}\simeq\overline{\KK}_n*\overline{\KK}_m$.
\end{proof}

\begin{cor} \label{Hccom} Let $G$ be a simple graph and let $\overline{\KK}_m$ denotes the complement of $\KK_m$. Then
$$H_{\mathbb{K}[G*\overline{\KK}_m]}(t)=H_{\mathbb{K}[G]}(t)+\frac{1}{(1-t)^m}-1.$$
In particular,
$$H_{\mathbb{K}[\KK_{m,n}]}(t)=\frac{1}{(1-t)^n}+\frac{1}{(1-t)^m}-1.$$
\end{cor}

\begin{rem} Let $\KK_{n_1,\ldots,n_r}$ denotes the complete multipartite graph with vertex partition $\cup_{i=1}^rV_i$, where $|V_i|=n_i$ for all $1\leq i \leq r$. Then
$$H_{\mathbb{K}[\KK_{n_1,\ldots,n_r}]}(t)=\sum_{i=1}^r\frac{1}{(1-t)^{n_i}}-(r-1).$$
\end{rem}

As the final result of this section we will compute Hilbert series of $G*\mathcal{S}_m$. To do this we just need to notice that
$$G*\mathcal{S}_m\simeq G*\overline{\KK}_m*\KK_1,$$
and apply Corollary \ref{Hilbert:graph}, or Corollaries \ref{HK1} and \ref{Hccom}.

\begin{cor} Let $G$ be a simple graph and $\mathcal{S}_m$ be the star graph with $m+1$ vertices. Then
$$H_{\mathbb{K}[G*\mathcal{S}_m]}(t)=H_{\mathbb{K}[G]}(t)+\frac{1}{(1-t)^m}+\frac{1}{1-t}-2.$$
\end{cor}

\providecommand{\bysame}{\leavevmode\hbox
to3em{\hrulefill}\thinspace}

\end{document}